\documentclass{article}  
\usepackage{url}
\usepackage{physics} 
\usepackage{amsmath}
\usepackage{amssymb,setspace,amsthm,calrsfs}
\usepackage{pgfplotstable,pstool,psfrag}
\usepackage{float}
\usepackage{tabularx}
\usepackage{ltablex}
\usepackage{booktabs}
\usepackage{tabularx}
\usepackage{supertabular}
\usepackage{tikz-cd}
\usepackage{siunitx}
\usepackage{stackengine}
\usepackage{xcolor}
\usepackage{color}
\definecolor{myblue}{RGB}{10,188,199}
\usepackage{acronym}

\usepackage{hyperref}
\hypersetup{colorlinks=true,linkcolor=myblue,urlcolor=myblue,citecolor=myblue,}
\usepackage{url}
\stackMath

\newcommand{\R}{\mathbb{R}}
\newcommand{\N}{\mathbb{N}}

\newcommand{\probability}{\mathbb{P}}

\def\defeq{\mathrel{\mathop:}=}
\newtheorem{thm}{Theorem}
\newtheorem{definition}[thm]{Definition}

\newtheorem{remark}[thm]{Remark}
\newtheorem{proposition}[thm]{Proposition}

\makeatletter
\let\mytagform@=\tagform@
\def\tagform@#1{\maketag@@@{\color{myblue}(#1)}}
\makeatother

\usepackage{graphicx}
\begin{document}
\title{On the Ergodic Control of Ensembles \\ in the Presence of Non-linear Filters}
\author{V. Kungurtsev, J. Marecek, R. Ghosh, R. Shorten}
\maketitle

\begin{abstract}                          
In many sharing-economy applications, as well as in conventional economy applications, one wishes to regulate the behaviour of an ensemble of agents with guarantees on both the regulation of the ensemble in aggregate and the revenue or quality of service associated with each agent.
Previous work [Automatica, Volume 108, 108483] has developed guarantees of unique ergodicity when there are linear filters. Here, we extend the guarantees to systems including non-linear elements, such as non-linear filters. 
\end{abstract}

\section{Introduction}

In an increasing number of sharing-economy applications, one wishes to regulate the stochastic behaviour of an ensemble of agents \cite{marevcek2016r,ErgodicControlAutomatica,8814786,griggs2021unique,marecek2021predictability,Quan2022}. 
For example, in a virtual power plant \cite{marecek2021predictability}, one may wish to regulate an ensemble consisting of distributed generation (rooftop photovoltaic panels) and interruptible loads (charging of electric vehicles), such that the overall load can be kept close to a constant over the imbalance settlement period 
within an electricity market.  Each agent enrolled in such a
demand-response aggregator 
provides -- at best -- probabilistic guarantees on its own response to a signal.
At the same time, the agent would like to have some guarantees on its revenue (from the distributed generation) or quality of service (for interruptible loads, such as charging), conditional on its responses. Many other two-sided markets and multi-sided platforms \cite{griggs2021unique} need to be balanced in the sharing economy, including networks of electric vehicles \cite{crisostomi2014plug,crisostomi2017electric}, ride hailing \cite{griggs2021unique}, and room rentals.
More broadly, much of the reasoning about fairness in artificial intelligence \cite{Quan2022} can be cast in this fashion.

Our work builds on a number of recent results in the area of the use of iterated random functions (IRF, see, e.g. \cite{Diaconis1999}) for modeling the closed loop involving an ensembles of agents competing to acquire or use a shared resource \cite{marevcek2016r,ErgodicControlAutomatica,griggs2021unique,marecek2021predictability}. The motivation for using iterated random functions in such applications is many-fold. First, IRF's have potential to model agent-level behaviour in many man-made systems. For example, IRF's are intuitively appealing for modelling ensembles of rational agents
orchestrated by means of a feedback signal such as a price, and extensions thereof towards rational inattention. Intuitively, it is appealing to consider each agent picking one of pre-defined behaviours at random, conditional on a state.
 Notably, using iterated random functions allows us to deal with systems that are not quite large-scale enough to be considered a fluid, but are sufficiently large-scale to prohibit traditional microscopic modeling. 
 Second, the use of IRF is very convenient from a modelling perspective. Surprisingly strong results exists for IRF's, and characterising the predictable and fair behaviour of shared-economy systems in terms of properties of a unique invariant measure is very natural.
Finally, the use of IRF in a control setting allows for non-trivial insights: in \cite{ErgodicControlAutomatica}, the difficulty in the use of PI control for regulating access in such settings is discussed. In particular, the loss of ergodicity when using controllers with poles on the unit circle is established. 
In \cite{marecek2021predictability}, these phenomena are demonstrated in power-systems applications using standard simulators.
Other uses of IRF in a systems-and-control context include the analysis of switched and hybrid systems \cite{664150,8814786}.

\begin{figure}[t!]
\centering
\includegraphics[width=0.92\columnwidth]{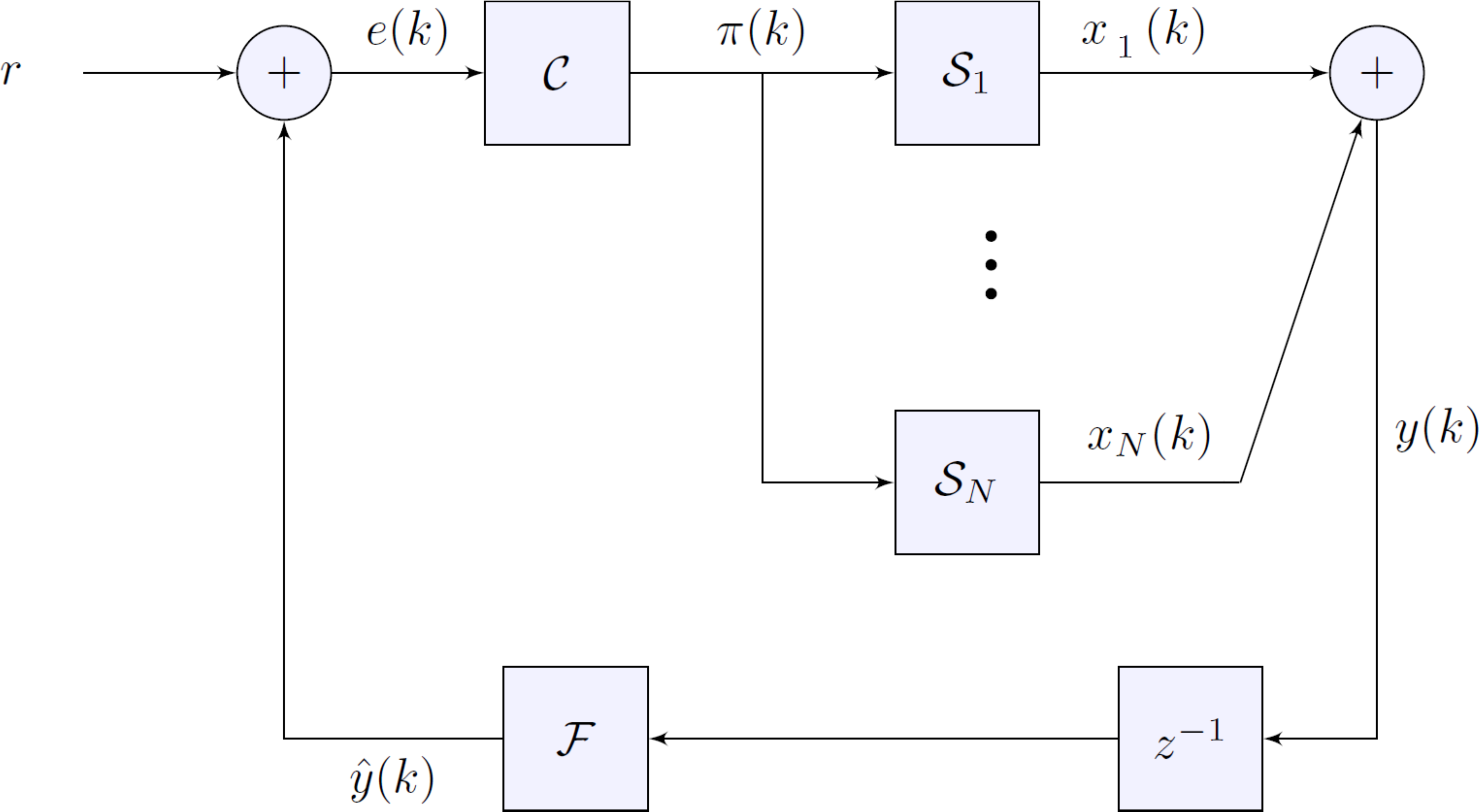}
\caption{A feedback model employed in \cite{ErgodicControlAutomatica} and here. $\mathcal{C}$ denotes a controller \eqref{eq:non-linear-cont}; $\mathcal{F}$ denotes a filter \eqref{eq:non-linear-filter}; $\mathcal{S}_i$ denotes agent $i$. 
}
\label{fig:system}
\end{figure}

The work presented here directly extends our original work in \cite{ErgodicControlAutomatica} by considering classes of \emph{nonlinear} feedback architectures. 
In particular, we develop guarantees for the existence of a unique invariant measure in closed loops consisting of nonlinear controllers of ensembles with nonlinear filters, in contrast to the earlier work which restricted itself to linear controllers. This presents significant methodological challenges. Specifically, the techniques establishing ergodicity must be developed to accommodate the wider range of behavior of nonlinear control systems. A linear system has many properties that are invariant across the control and state space, which is distinctly not the case for nonlinear controllers, whose local and global behavior can be distinct. Formally, ergodicity of IRF is straightforward to establish under iterate-to-iterate contraction guarantees, which naturally coincide with eigenspectrum-associated properties of linear systems.

The techniques establishing ergodicity in this work are founded on appropriately combining the seminal work on stability of nonlinear systems of Lohmiller and Slotine~\cite{Slotine1998} together with tools for establishing ergodicity in IRF, e.g., as found it~\cite{meyn2009markov}. This requires a fundamentally different consideration of the random dynamics. Whereas for linear systems, behavior is averaged at every iteration and the individual distinctions across the functional maps can be abstracted away, for nonlinear systems, we rely on the concept of \emph{canonical trajectories} that represent stable solution paths of each process by itself, and arguments must carefully incorporate how they could interact.


\section{The Closed Loop as a State-Dependent IRF}\label{s:irf}
Let $\mathcal K$  be a closed subset  of $\mathbb R^n$ with the usual Borel
$\sigma$-algebra $\mathcal{B}\left(\mathcal K\right)$ consisting of all possible events. We consider a discrete-time-homogeneous Markov chain on $\mathcal K$ 
with the transition operator $\mathbf{P}$ as follows:
for $x \in \mathcal K$, $\mathcal{G} \in \mathcal{B}\left(\mathcal K\right)$
\begin{align}\label{eq:transi-op}
\mathbf{P}(x,\mathcal{G}) := \mathbb P \left(X_{k+1} \in \mathcal{G} \; \vert \; X_k = x\right).
\end{align}
Conditioned on an initial distribution $\xi$, the random variable $X_k$ is distributed
according to the probability measure $\xi_k$ which is determined inductively by
\begin{equation}\label{eq:measure-iteration}
\xi_{k+1}(\mathcal{G})  := \int_{\mathcal K} \mathbf{P}(x, \mathcal{G}) \, \xi_k\left(\mathrm{d} x\right), \text{for all } \mathcal{G} \in \mathcal{B}\left(\mathcal K\right).
\end{equation}
If there exists a probability measure $\nu$ on $\mathcal K$ which is a fixed point for the iteration scheme described by the equation \eqref{eq:measure-iteration}, then $\nu$ is called
an invariant probability measure with respect to the Markov process $\{ X_k \}_{k\in \mathbb N}$. Thus, $\nu$ is invariant if
\begin{align}\label{eq:inv-prob-meas}
\left(\mathbf{P} \nu\right) \left(\mathcal G\right) = \nu \left(\mathcal{G}\right)\quad \forall \mathcal G\in \mathcal B\left(\mathcal K\right).    
\end{align}
Further,  $\nu$ is called
attractive, if for all probability measure $\mu$ the sequence $\{
\xi_k \}$ defined by \eqref{eq:measure-iteration} with initial
condition $\mu$ converges to $\nu$ in distribution.\newline

We assume there are $w_i \in \N$ state transition maps ${\mathcal W}_{ij}: \R^{n_i} \to \R^{n_i}$, $j=1,\ldots,w_i$
for agent $i$
and $h_i \in \N$ output maps ${\mathcal H}_{i\ell}: \R^{n_i} \to
\mathbb{D}_i$, $\ell= 1,\ldots,h_i$ for each agent $i$, where $\mathbb{D}_i$ is the space defining agent $i$'s behaviors (demands, in the shared economy case).
The evolution of the states and
the corresponding demands then satisfy:
\begin{align}\label{eq:general-dem}
x_i(k+1) & \in  \{  {\mathcal W}_{ij}(x_i(k)) \;\vert\; j = 1, \ldots, w_i\}, \\
y_i(k) & \in  \{ {\mathcal H}_{i\ell}(x_i(k)) \;\vert\; \ell = 1, \ldots, h_i\},
\end{align}
where the choice of agent $i$'s response at time $k$ is governed by probability
functions $p_{ij} : \Pi \to [0,1]$, $j=1,\ldots,w_i$, respectively
$p'_{i\ell} : \Pi \to [0,1]$, $\ell=1,\ldots,h_i$, where $\Pi\subseteq \mathbb{R}^c$ is a set of control signals, which themselves are governed by state dependent function $\pi:\mathcal{X}\to \Pi$. For brevity we use $\pi(x(k))$ and $\pi(k)$ interchangeably and denote $\pi$ as both an element of $\Pi$ and a function whose domain is $\Pi$ depending on context. Specifically, for each agent $i$, we have for all $k\in\N$ that
\begin{subequations} \label{eq:prob-laws}
\begin{align}
&\mathbb{P}\big(x_i(k+1) = {\mathcal W}_{ij}(x_i(k)) \big) = p_{ij}(\pi(k)),\\
&\mathbb{P}\big( y_i(k) = {\mathcal H}_{i\ell}(x_i(k)) \big) = p'_{i\ell}(\pi(k)).
\label{eq:prob-law-c}
\intertext{Additionally, for all $\pi \in \Pi$, $i=1,\ldots,N$ it holds that}
&\sum_{j=1}^{w_i} p_{ij}(\pi) = \sum_{\ell=1}^{h_i} p'_{i\ell}(\pi) = 1.
\end{align}
\end{subequations}
Conditioned on $\{ x_i(k)  \}, \pi(k)$, the random variables $\{ x_i(k+1) \mid  i = 1,\ldots,N \}$ are stochastically independent. The outputs $y_i(k)$  depend on $x_i(k)$ and the
signal $\pi(k)$ only for all $k$. We shall denote the stochastic realization of the selected map as $\xi(k)=\xi(x(k))\in\Xi$.

We can stack the agents' state vectors together to define the filter and control system,
\begin{equation}\label{eq:non-linear-filter}
{\mathcal F} ~:~ \left\{\begin{array}{ccl}
x_f(k+1)  &=& {\mathcal W}_{f}(x_f(k),y(k))\\
\hat y(k) &=& {\mathcal H}_{f}(x_f(k),y(k)),
\end{array} \right.
\end{equation}
\begin{equation}\label{eq:non-linear-cont}
{\mathcal C} ~:~ \left\{ \begin{array}{ccl}
x_c(k+1) &=& {\mathcal W}_{c}(x_c(k),\hat y(k),r) \\
\pi(k) &=& {\mathcal H}_{c}(x_c(k),\hat y(k),r),
\end{array} \right.
\end{equation}
If we denote by $\mathbb{X}_i, i=1,\ldots,N, \mathbb{X}_{\mathcal C}$ and $\mathbb{X}_{\mathcal F}$ the state spaces of the
agents, the controller and the filter, then the system evolves on the
overall state space $\mathbb{X} := \prod_{i=1}^N \mathbb{X}_i \times \mathbb{X}_{\mathcal C} \times \mathbb{X}_{\mathcal F}$
according to the dynamics
\begin{equation}\label{eq:general-dynamics}
x(k+1) := \begin{pmatrix}
(x_i)_{i=1}^N  \\
x_f \\
x_c
\end{pmatrix} (k+1) \in \{ F_m(x(k)) \,\vert\, m \in {\mathbb M}\}.
\end{equation}
where each of the maps $F_m$ is of the form
\begin{equation}\label{eq:F_m-definition}
F_m(x(k)) \defeq  \begin{pmatrix}
( {\mathcal W}_{ij}(x_i (k)) )_{i=1}^N  \\
{\mathcal W}_f(x_f(k), \sum_{i=1}^N  {\mathcal H}_{i\ell}(x_i (k))) \\
{\mathcal W}_c(x_c(k), {\mathcal H}_f(x_f(k), \sum_{i=1}^N  {\mathcal H}_{i\ell}(x_i (k))))
\end{pmatrix}
\end{equation}
and the maps $F_m$ are indexed by indices $m$ from the set
\begin{equation}\label{eq:all-maps-index-set}
\mathbb{M} \defeq \prod_{i=1}^N \{ (i,1), \ldots, (i,w_i) \} \times \prod_{i=1}^N \{ (i,1), \ldots, (i,h_i) \}.
\end{equation}
By the independence assumption on the choice of the transition maps and
output maps for the agents, for each multi-index
$m=((1,j_1),\ldots,(N,j_N),(1,l_1),\ldots,(N,l_N))$ in this set, the
probability of selecting the corresponding function $F_m$ is given by
\begin{multline}\label{eq:prob-chosing-F_m}
\probability \left( x(k+1)=F_m(x(k)) \right) = \\ \left(\prod_{i=1}^N
p_{ij_i}(\pi(k)) \right) \left( \prod_{i=1}^N p'_{il_i}(\pi(k))
\right) =: q_m(\pi(k)).
\end{multline}
Let us denote $\sigma(k,x(0),\{\xi(j)\}_{j=1..k}):\mathbb{N}\times\mathbb{X}\times \Xi^k\to \mathbb{N}$ to indicate that $\sigma(k,x(0),\{\xi_j\})=l$ if the evolution satisfies $x(k+1)=F_l(x(k))$.

Clearly, we have obtained an \emph{iterated random function with state dependent prob
abilities}.

\section{Stability of Nonlinear Systems and Canonical Trajectories}
Each composite map $F_m$ takes $x(k)$ to $x(k+1)$ in a deterministically defined way, i.e., it corresponds to a particular filter and control action given the state $x(k)$. As such it can be associated to a dynamic system in its own right, and we define some classical notions of stability for these systems. In the sequel we shall see that careful agglomeration of the associated quantities measuring and confirming stability for these maps shall endow ergodic properties of the composite IRF. We denote $F_m^{(k)}(x(0))$ to be $k$ repeated applications of $F_m$ starting from $x(0)$.

We present the definitions and propositions necessary for our development. The notions are classical in the literature, e.g.~\cite{Slotine1998} and we leave the reader to consult references for definitions of specific technical notions, such as $\mathcal{L}$ and $\mathcal{K}\mathcal{L}$ functions. In addition, we write $F^{(k)}_m(x) = \phi_m(k,x,u_m(\cdot))$ to indicate that the map $F_m(x)$ physically corresponds to a controlled system $\phi_m$ with control law $u_m(x)$. We adjust the standard definitions accordingly.
A significant property of a deterministic system is asymptotic incremental stability.\newline

\begin{definition}(Uniform Asymptotic Incremental
Stability)~\cite{Slotine1998} The system $F_m$ is uniformly asymptotically incrementally
stable in a positively invariant set $\mathcal X$ if there exists $\beta \in \mathcal K \mathcal L$ such that for any pair of initial states $x_1(0), x_2(0) \in \mathcal X$, for all $k\ge k_0$ for some $k_0>0$, the following holds
\begin{align}
& \|F_m^{(k)}(x_1(0))-F_m^{(k)}(x_2(0))\|\\ &\le \|\phi (k,x_1(0), u(\cdot))-\phi (k,x_2(0), u(\cdot))\|\nonumber\\
&\le \beta (\|x_1(0)-x_2(0)\|, \|u(x_1(0))-u(x_2(0))\|_{\infty}).   
\end{align}
If $\mathcal X= \mathbb R^n$, then the system  $F_m$ is called uniformly globally asymptotically incrementally stable.
\end{definition}
It is common that demonstrating the existence of a Lyapunov function is a technique used to demonstrate a system's stability. Here, we note the underutilized fact that the reverse also holds as well.\newline 

\begin{proposition}\label{prop:lyaptostable}\cite{jiang2002converse}
If system $F_m(x)$ is uniformly globally asymptotically incrementally stable then there exist a smooth function $V_m:\mathbb R^d \to \mathbb R_{\ge 0}$, functions $\alpha_1,\alpha_2\in \mathcal K_{\infty}$ and $\alpha_3: \mathbb R_{\ge 0}\to \mathbb R_{\ge 0}$ positive definite such that there exists $x^{(m)}(k)$ such that
for all $x_1,x_2\in \mathbb R^n$ and trajectories of $F^{(k)}_m(x(0))$ denoted by $x(k)$ the following holds:
\begin{align}\label{cond1}
\alpha_{1}\left(\|x_{1}-x_{2}\|\right)\leq V_m\left( x_{1}- x_{2}\right)\leq \alpha_{2}\left(\|x_{1}-x_{2}\|\right),
\end{align}
\begin{align}\label{cond2}
&V_m\left(F_m(x(k))\right)-V_m\left(x(k)\right)\nonumber\\
&\leq-\alpha_{3}\left(\|F_m(x(k))-x^{(m)}(k)\|\right).
\end{align}
\end{proposition}

We shall define $x^{(m)}$ as the \emph{canonical stable trajectory} associated with map $F_m$. 

In consideration of the IRF, we shall also make use of the notion of a stochastic Lyapunov function \cite{meyn1989ergodic}, and the associated condition for ergodicity.\newline
\begin{proposition}\cite{meyn2009markov}
Let $\sigma(k)$ be the stochastic process of transitions for a trajectory of IRF $x(k)$. If there exists a stochastic Lyapunov function $V(k,x(0),\sigma)$ such that,
\[
\mathbb{E}\left[V(k+1,x(0),\sigma)|\mathcal{F}_k\right]-V(k,x(0),\sigma)\le -K<0
\]
then the system is ergodic,
where $\mathcal{F}_k$ is the associated filtration of the IRF.\newline
\end{proposition}

Finally, we will consider contractive systems, which guarantee stability for deterministic processes.
Let $M_n(\mathbb R)$ denote the set of all $n\times n$ matrices with real-entries. Given $P,Q\in M_n(\mathbb R)$ be symmetric, we write $P \preceq Q$ if and only if $x^\top Px \le x^\top Qx\quad \forall x\in \mathbb R^n$.\newline

\begin{definition}(Uniform Contraction)
Assume that $F_m$ is continuously differentiable in $x$ on a positively invariant set $\mathcal{X}\subseteq \mathbb R^n$. Then the system is uniformly contracting in $\mathcal X$ if there exist a non-singular matrix-valued function $\psi_m: \mathbb R^d \times \mathbb R^n\to M_n(\mathbb R)$ and constants $\mu, \eta, \rho\in \mathbb R_{+}$ such that for all $x(k)\in \mathbb R^n$, we have 
\begin{align}\label{def:unicontra1}
\eta I \preceq \psi_m(u(x(k)), x(k))^{\top} \psi_m(u(x(k)), x(k)) \preceq \rho I,
\end{align}
\begin{align}\label{def:unicontra2}
\mathcal M_m(u(x(k)), x(k))^\top \mathcal M_m(u(x(k)), x(k))-I \preceq-\mu I,
\end{align}
where the matrix $\mathcal M_m(u, x)$ is defined as 
\begin{align}\label{contramat}
&\mathcal M_m(u, x)\nonumber\\
&=\psi_m(u(F_m(x)), F_m(x)) \frac{\partial F}{\partial x}(u(x), x) \psi(u(x), x)^{-1}.    
\end{align}
If $\mathcal X=\mathbb R^n$, then we  the system $F_m$ is uniformly globally contracting.
\end{definition}
Uniform global contraction implies global exponential stability~\cite[Theorem 1, Theorem 3]{Slotine1998}.



\section{Main Result}

For a linear controller and filter, we have established \cite{ErgodicControlAutomatica} conditions that assure the existence of a unique invariant measure for the closed loop. This is often referred to as the  
the ``unique ergodicity'' property in the language of \cite{ErgodicControlAutomatica}. 
Let us now consider non-linear controllers and non-linear filters and establish conditions that ensure unique ergodicity of the feedback system in Figure \ref{fig:system}. We present our first Theorem on unique ergodicity as guaranteed by global stability properties of the nonlinear control systems $\{F_m(x)\}$.\newline 

\begin{thm}\label{thm:unified-stability}
Consider the closed loop stochastic feedback system described in Section~\ref{s:irf}. Furthermore, assume:
\begin{itemize}
\item[(i)] we have globally Lipschitz-continuous and continuously differentiable functions ${\mathcal W}_{ij}$ and ${\mathcal H}_{ij}$ such that each $F_m$ is globally Lipschitz with constant $l_m$;\newline 
\item[(ii)] there are  scalars $\delta, \delta' > 0$ such that
$p_{ij}(\pi) \geq  \delta > 0$,
$p'_{ij}(\pi) \geq  \delta' > 0$ for all $\pi\in \Pi$ and all $(i,j)$. Furthermore, for all $\pi\in\Pi$ the Markov matrix induced by $p_{ij}$ is irreducible. \newline 
\item[(iii)] the set $\Pi$ is compact;\newline 
\item[(iv)] for all $m$, $F_m(x(k))$ is Uniformly Globally Asymptotically Stable and satisfies one of the following conditions: \newline 
\begin{itemize}
\item[(a)] $\max_m l_m (2-|\mathbb{M}|\delta)<1$;\newline 
\item[(b)] For all $m$, for all $x,x'\in\mathcal{X}$ it holds that $\|\pi(F_m(x))-\pi(F_m(x'))\|\le \kappa \|x-x'\|$ with $\kappa<1$, i.e., $\pi\circ F_m$ is a contraction.\newline 
\end{itemize}
\end{itemize}
Then, the feedback loop has a unique, attractive invariant measure. In particular, the system is uniquely ergodic.
\end{thm}
\begin{proof}
By the compactness of $\Pi$, there exists a limit point
$\pi^{\star}$ of the sequence $\pi(k)$ generated by the IRF. By the irreducibility of $p_{ij}(\pi^{\star})$, it has an invariant distribution $\nu^{\star}$. 
Assumption (iv) and Proposition~\ref{prop:lyaptostable} imply the existence of a set $\{x^{\left(m\right)}(k)\}$ of canonical convergent trajectories together with Lyapunov function sequences $V_m$ satisfying
\begin{align}\label{eq:lyap-diff-cond}
V_m(k+1,F(x(k)))-V_m(k,x(k)) \\ \le -\alpha_m(\|x(k)-x^{\left(m\right)}(k)\|). \notag
\end{align}
Now define the stochastic Lyapunov function \cite{meyn1989ergodic},
\begin{align}
V(k,x,\sigma) = \sum\limits_{m\in\mathbb{M}} \mathbf{1}[\sigma(k,x,\cdot)=m]
V_m(k,x).
\end{align}\label{eq:stoch-lyap-func}
Consider a trajectory for which 
\[
x(k+1)= F_{\tilde\sigma(k,x(k),\{\xi_j\})}(x(k)),
\]
with $\tilde \sigma$ is such that $\pi(k)=\pi^{\star}$ regardless of $x(k)$. It holds that,
\[
\begin{array}{l}
\mathbb{E}\left[V(k+1,x(k+1),\tilde\sigma)|\mathcal{F}_k\right] \\
\quad = \mathbb{E}\left[\sum\limits_{m\in\mathbb{M}} \mathbf{1}[\tilde\sigma(k,x(k))=m]
V_m(k+1,x(k+1))|\mathcal{F}_k\right]
\\ \quad= 
\sum\limits_{m\in\mathbb{M}} p^{\star}_{\tilde\sigma(k-1),m}
V_m(k+1,x(k+1)) \\
\quad \le \sum\limits_{m\in\mathbb{M}} p^{\star}_{\tilde\sigma(k-1),m}
\left(V_m(k,x(k))-\alpha_m(\|x(k)-x^{\left(m\right)}(k)\|)\right).
\end{array}
\]
Taking expectations, recalling the stationary distribution $\nu^*$, we obtain that,
\begin{align*}
\mathbb{E}[V(k,x,\nu)] = \sum\limits_{m\in\mathbb{M}}
\nu^*_m V_m(k,x)    
\end{align*}
and,
\begin{align*}
&\mathbb{E}\left[V(k+1,x(k+1),\tilde\sigma(k))\right]\\
&\le \sum\limits_{m\in\mathbb{M}} \nu^*_m
\left(V_m(k,x(k))-\alpha_m\left(\left\|x(k)-x^{\left(m\right)}(k)\right\|\right)\right)\\
&=\mathbb{E}\left[V(k,x,\tilde\sigma(k-1))\right] - \sum\limits_{m\in\mathbb{M}} \nu_m\alpha_m \left(\left\|x(k)-x^{\left(m\right)}(k)\right\|\right)    
\end{align*}
which implies that $x(k)$ driven by the process $\{F_{\tilde \sigma(k)}(x(k))\}$ is stable in probability and that the IRF is ergodic. 

We now consider the IRF as proceeding with general $\pi(k)$, i.e., not fixed at $\pi^*$. We divide the argument into two cases as depending on which of condition iv above holds.

\textbf{Case (iv-b)}

$\pi\circ F_m$ as a map from $\Pi$ to $\Pi$ is a contraction. Then by Kakutani's fixed point theorem \cite{Kakutani1941}, $\pi_k$ must converge to $\pi^*$ above and thus asymptotically $\sigma$ behaves as $\tilde \sigma$.

\textbf{Case (iv-a)} 
Now consider a contraction argument involving two stochastic process paths of $x(k)$ for the non-linear IRF, $X(k)$ and $\tilde X(k)$, denoting $\tilde x(k+1,\pi(X(k)))$ the stochastic step from $\tilde X(k)$ to $k+1$ governed by the probability mappings of $X(k)$ (i.e., $\pi(X(k))$): 
\[
\begin{array}{l}
W_2 (X(k+1),\tilde X(k+1))\\ 
\le W_2 (X(k+1,\pi(X(k))),\tilde x(k+1,\pi(X(k))))\\ \qquad +W_2 (\tilde x(k+1,\pi(X(k))),\tilde X(k+1,\pi(\tilde X(k)))) \\ 
\le \sum\limits_m \mathbb{E}_{X(k)}l_m \pi_m(X(k))
W_2(X(k),\tilde X(k)) \\ 
\qquad + \max_m l_m (1-|\mathbb{M}|\delta) W_2(X(k),\tilde X(k)) \\
\le  \max_m l_m (2-|\mathbb{M}|\delta) W_2(X(k),\tilde X(k))
\end{array}
\]
where $W_2$ is Wasserstein-2  distance \cite{Villani2009} and the first inequality is just the triangle inequality. For the second inequality, for the first term the same $\pi(X(k))$ governs the probabilistic function selection and so the distances of the next iterate is the weighted, by the probabilities in $\pi(X(k))$, average over the corresponding distances for each of the $F_m$, which are its Lipschitz constants. For the second term, considering the Wasserstein-2 distance is the distance corresponding to the optimal transference plan, recall by the assumptions that only for at most, $1-|\mathbb{M}|\delta$ of the probability mass of $\pi$ can result in a different map $m$ to apply at $k$ for $\pi(X(k))$ compared to $\pi(\tilde X(k))$, and if indeed a different map is chosen the distances of the two outputs are bounded as according to the Lipschitzianness of each $F_m$ with a bound $\max_m l_m$.
The coupling contraction now implies ergodicity.
\end{proof}
It can be observed that the conditions are relatively strong, but they are the first and thus far the weakest, to the best of our knowledge, to establish unique probabilistic convergence in the nonlinear setting and as such can be considered the start of a research program. One of the main considerations for future work is in particular to relax the global Lipschitzianity conditions.

In addition, however, if the individual systems $F_m$ satisfy a uniform stability condition, guarantees could be shown in terms of a boundedness of the trajectory, without quantitative contraction assumptions.\newline 

\begin{thm}\label{thm:Lohmiller-Slotine-typ}
Consider the closed loop feedback system in Figure \ref{fig:system} and formalized in Section~\ref{s:irf}. Assume that:
\begin{itemize}
\item[(i)] There exists a $\beta$, a region $\mathcal{X}$, with $x(0)\in\mathcal{X}$, and, for all $m$, there exists a trajectory $x^{\left(m\right)}(k)$, a positive definite metric $M^{\left(m\right)}=\Theta_m^T \Theta_m$ centered at  $x^{\left(m\right)}(k)$ such that for all $x\in\mathcal{X}$,
it holds that for $G_m=\Theta_m^T \frac{\partial F_m(x)}{\partial x} \Theta_m$, 
\begin{align}\label{eq:stable-trajec}
G_m^T G_m-I\preceq -\beta I \prec 0
\end{align}
\item[(ii)] There exists some $R$ such that
\begin{align}\label{eq:bound-R}
\sup\limits_k \max\limits_{m_1,m_2}\left\|x^{\left(m_1\right)}(k)- x^{\left(m_2\right)}(k)\right\|\le R< \infty.    
\end{align}
\end{itemize} 
Then, there exists some $D>0$ such that for all $x(0)$ and $\epsilon$, 
\begin{equation}
\lim\limits_{k\to\infty}\mathbb{P}\left[ 
\mathrm{dist}\left(x(k),\bar{X}(k)\right)\le D+\epsilon \right] = 1
\end{equation}
where 
\begin{equation}\label{eq:trajec-hull}
\bar X(k) := \left\{\sum\limits_{m\in \mathbb{M}} \alpha_m  x^{\left(m\right)}(k):\, \sum\limits_{m\in \mathbb{M}} \alpha_m=1\right\}
\end{equation}
is the convex hull of all stable trajectories of $\{F_m\}$, $m\in\mathbb{M}$ and $\mathrm{dist}(y,\bar X)$ is
the projection onto $\bar X$ in the $L_2$ norm, i.e., $\min_{x\in\bar X} \|x-y\|$.
\end{thm}

\begin{remark}
Following \cite[Section 5]{Slotine1998} the condition (i) above
ensures that all neighbouring trajectories converge, i.e., the system dynamics exhibit global exponential convergence to a single trajectory.
\end{remark}
\begin{proof}
Consider a set of trajectories $\{\tilde x^{\left(m\right)}(k)\}$ generated by repeated application of the map $F_m(\cdot)$. 
By (i) and \cite[Theorem 3]{Slotine1998}, we have that any potential trajectory
\begin{align*}
\tilde x^{\left(m\right)}(k+1)=F_m(\tilde x^{\left(m\right)}(k))    
\end{align*}
starting in $\mathcal{X}$ converges exponentially to $x^{(m)}(k)$. Thus, there exists a $\lambda_m<1$ such that
\begin{align}\label{eq:expo-convrg-1}
\left\|x^{\left(m\right)}(k+1)-\tilde x^{\left(m\right)}(k+1)\right\|\le \lambda_m \left\|x^{\left(m\right)}(k)-\tilde x^{\left(m\right)}(k)\right\|    
\end{align}
and for any $x$, 
\begin{align}\label{eq:expo-convrg-2}
\left\|x^{\left(m\right)}(k+1)-F_m(x)\right\|\le \lambda_m\left \|x^{\left(m\right)}(k)-x\right\|.  
\end{align}
Since $\mathbb{M}$ is finite, there exists $\lambda=\max_m \lambda_m <1$. 

Consider any general trajectory $x(k)$ generated by the IRF $\{F_{\sigma_k}(\cdot)\}$. Note that,
\[
\left\|F_{\sigma_k}(x(k))-x^{\left(\sigma_k\right)}(k+1)\right\| 
\le \lambda \left\|x(k)-x^{\left(\sigma_k\right)}(k)\right\|,
\]
and so, since $x(k+1)=F_{\sigma_k}(x(k))$ and now using (ii),
\[
\begin{array}{l}
\mathrm{dist}\left(x(k+1),\bar X(k+1)\right) \le
\left\|F_{\sigma_k}(x(k))-x^{\left(\sigma_k\right)}(k+1)\right\| \\ \quad \le
\lambda \left\|x(k)-x^{\left(\sigma_k\right)}(k)\right\| \le 
\lambda \left(\mathrm{dist}(x(k),\bar{X}(k))+R\right)
\\ \quad \le  \lambda\hspace{0.05 cm} \mathrm{dist}\left(x(k),\bar{X}(k)\right)+\lambda R, 
\end{array}
\]
which now implies, for all possible IRF realizations $\{\sigma_k\}$,
\[
\begin{array}{l}
\mathrm{dist}\left(x(k+1),\bar X(k+1)\right) 
\\ \quad \le \frac{\lambda R}{(1-\lambda)}+\left(\mathrm{dist}(x(0),\bar{X}(0))-\frac{\lambda R}{(1-\lambda)}\right)\lambda^k,
\end{array}
\]
and the Theorem holds with $D=\frac{\lambda R}{(1-\lambda)}$.
\end{proof}
\section{Illustrations and Discussion}

Let us now return to the example of the control of the ensemble mentioned in the introduction. 
There \cite{Schlote2014,crisostomi2014plug}, we have a fleet of plug-in hybrid vehicles. The vehicles randomly select a mode of operation (electric, non-electric) at each time step based on a feedback signal (pollution levels). 
PI control is commonly used to regulate the aggregate air pollution emissions produced by the fleet, turning the internal combustion engines on and off.
The filter considered is also non-linear, considering the maximum of the emissions recorded within a time window. This filter \cite{marevcek2016r} is clearly non-linear, and only marginally stable in the realization of \cite[Step 2 of the proof of Theorem 1]{marevcek2016r}.  
The negative results of \cite[Theorem 6]{ErgodicControlAutomatica} could be readily extended to show that the system lacks unique ergodicity.
The positive results of \cite[Theorem 16 ]{ErgodicControlAutomatica} are restricted to stable linear filters, however, and hence do not apply.
When we apply the lag approximant of the PI control, it is easy to see that conditions of Theorem \ref{thm:unified-stability} are satisfied. 
Notably, 
Condition (i) is satisfied, when we consider some proportion of the use of the internal combustion engine within a unit of time. (See \cite[Theorem 18]{ErgodicControlAutomatica} for  $F_m$ that are not continuous.)
Condition (iv-b) is satisfied by the cascade of the stable lag approximant and marginally stable filter.

In the margin, let us mention that one could also use incremental input to state stability, as suggested by   \cite{marecek2021predictability}.

We believe our results may be of use in the context of many further smart-city applications, for a number of reasons. Often, in such applications we often use a feedback signal (sometimes a price) to regulate access to a constrained resource of an ensemble of agents that (often) make binary decisions. In many such applications one employs an algorithm such as a PI control to regulate the price to orchestrate the behaviour of the ensemble towards some desirable outcome. Examples of such feedback loops can be readily found in the literature. For example in \cite{Griggs2016} agents stochastically bid for access to priority parking in a parking lot. In this application the feedback signal regulates the occupancy of the car park. Non-linearity may arise in the manner in which occupancy is filtered, or in order allocate access in a ``fair'' manner. 
Other examples include  transmission losses in power-systems applications \cite{marecek2021predictability} and the
Bureau of Public Roads (BPR) functions \cite{marevcek2016r}, 
which relate the travel time to the traffic flow in transportation research.

\section{Conclusions}
We have shown conditions under which non-linear filters allow for unique ergodicity within the control of ensembles. This has numerous applications in fields where the filters model non-linear physics of some underlying systems, such as water, transportation, and power systems. 

In a companion paper \cite{marecek2021predictability}, we show one more condition. In particular, in \cite[Section V, Theorem 2]{marecek2021predictability}, we show that for every incrementally input-to-state stable \cite{angeli2002lyapunov} controller and filter compatible with the feedback structure of Figure \ref{fig:system}, the feedback loop has a unique, attractive invariant measure and hence the system is uniquely ergodic.
Further conditions are left as a subject of further work. One could, for instance, consider further conditions following \cite{jiang2002converse} and using the arguments of ~\cite{tran2018convergence}, which have shown that certain stability and contraction properties are equivalent, or using conditions of \cite{Dragicevic}.

\paragraph{Acknowledgements}                               
R. Ghosh and R. Shorten were supported by Science Foundation Ireland grant 16/IA/4610.
V. Kungurtsev and J. Marecek have been supported by OP VVV project CZ.02.1.01/0.0/0.0/16\_019/0000765 ``Research Center for Informatics''.  
This work has received funding from the European Union’s Horizon Europe research and innovation programme under grant agreement no. 101070568.

\bibliographystyle{plain}        
\bibliography{autosam}           

                                        
\clearpage
\appendix

\end{document}